\documentclass{amsart}
\usepackage{eurosym}
\usepackage{amssymb}
\usepackage{amsfonts}
\usepackage{hyperref}

\hypersetup{
	hidelinks}
\newtheorem{theorem}{Theorem}
\theoremstyle{plain}

\numberwithin{equation}{section}

\begin{document}
\title[Singular minimal translation surfaces]{Singular minimal translation
surfaces in Euclidean spaces endowed with semi-symmetric connections}
\author{AYLA ERDUR$^{1}$}
\address{$^{1,3}$DEPARTMENT OF MATHEMATICS, FACULTY OF SCIENCE AND ART,
TEKIRDAG NAMIK KEMAL UNIVERSITY, TEKIRDAG 59100, TURKEY}
\author{MUHITTIN EVREN AYDIN$^{2}$}
\address{$^{2}$DEPARTMENT OF MATHEMATICS, FACULTY OF SCIENCE, FIRAT
UNIVERSITY, ELAZIG, 23200, TURKEY}
\author{MAHMUT ERGUT$^{3}$}
\email{aerdur@nku.edu.tr, meaydin@firat.edu.tr, mergut@nku.edu.tr}

\begin{abstract}
In this paper, we study and classify singular minimal translation surfaces
in a Euclidean space of dimension $3$ endowed with a certain semi-symmetric
(non-)metric connection.
\end{abstract}

\maketitle

\section{INTRODUCTION}

Let $\mathbb{R}^{3}$ denote a Euclidean space of dimension $3$  with the
canonical metric $\left\langle ,\right\rangle $ and $\left( x,y,z\right) $
the rectangular coordinates of $\mathbb{R}^{3}$. A surface $M$ in $\mathbb{R}%
^{3}$ is called \textit{translation surface} if it can be written as the sum
of two curves. Then, a local parametrization $\sigma $ on $M$
follows $\sigma \left( s_{1},s_{2}\right) =\gamma _{1}\left( s_{1}\right)
+\gamma _{2}\left( s_{2}\right) ,$ for $\gamma _{i}:I_{i}\subset \mathbb{%
R\rightarrow R}^{3}$, $i=1,2,$ \cite{Darboux}. In the special case\
that the curves $\gamma _{i}$ lie in orthogonal planes, up to a change of
coordinates, the surface $M$ can be locally expressed by the explicit
form $z=f(x)+g(y)$ for smooth functions $f,g$ of single variable. If such a surface is minimal, i.e. its mean curvature vanishes identically, then it is either a plane or describes the \textit{Scherk surface }\cite{Scherk}.

While this sort of surfaces has been studied in classical manner since the
former half of nineteenth century (see \cite{Aydin1}, \cite{Dillen1}-\cite%
{Dillen3}, \cite{Goemans}, \cite{Inoguchi}-\cite{Lopez3}, \cite{Lopez6}, \cite{Munteanu}, \cite{Scherk}-\cite%
{Yang2}, \cite{Yoon1}, \cite{Yoon2}), their complete classification/characterization in $\mathbb{R}^{3}$ emposing natural curvature conditions (e.g. minimality, flatness and having
nonzero constant mean or Gaussian curvatures) have been recently found, see \cite%
{HL1}-\cite{HL2}. Yet, in higher dimensions and different ambient spaces, there are still numerous unsolved problems.

On the other hand, a semi-symmetric metric (resp. non-metric) connection on a Riemannian manifold were defined by
Hayden \cite{Hayden} (resp. Agashe \cite{Agashe1}) and since then has been studied by many authors. For example, see \cite{Agashe2}, \cite{Akyol}, \cite{Chaubey}, \cite{De}, \cite{Dogru}, 
\cite{Gozutok}, \cite{Imai}, \cite{Murat}-\cite{Ozgur}, \cite%
{Yano1}, \cite{Yano2}, \cite{Yucesan2}.

As can be seen from the great number of published studies, the notions of translation surfaces and semi-symmetric (non-) metric connection have great interest and very recently Wang \cite{Wang} combined these seperate research areas into one, which came up with a new perspective. In the cited paper, the author introduced and obtained minimal translation surfaces in 3-dimensional space forms endowed with a certain semi-symmetric (non-)metric connection.

In this study we mainly concern with singular minimal surfaces in $\mathbb{R}^{3}$, namely those surfaces satisfying an equation of mean curvature type (see \cite{Dierkes2}). The notion of singular minimal surface is a generalization of two-dimensional analogue of the catenary which is known as a model for the surfaces with the lowest gravity center, in other words, one has minimal potential energy \cite{Bohme}. In this context, the present study aims to contribute to Wang's perspective by considering singular minimal translation surfaces in $\mathbb{R}^{3}$ endowed with a certain semi-symmetric (non-)metric connection.

In order to explicitly initiate the notion of singular minimality, we begin with the
one-dimensional case: let $\gamma :I\subset \mathbb{R\rightarrow R}^{2}$ be
a parametrized curve and $\mathbf{u}\in \mathbb{R}^{2}$ a fixed unit vector
and $\alpha $ some real constant. Then the curve $\gamma $ is called $\alpha
-$\textit{catenary} (see \cite{Dierkes2}) if the following holds 
\begin{equation}
\kappa \left( s\right) =\alpha \frac{\left\langle \mathbf{n}\left( s\right) ,%
\mathbf{u}\right\rangle }{\left\langle \gamma \left( s\right) ,\mathbf{u}%
\right\rangle },  \label{1.1}
\end{equation}%
where $\kappa $ and $\mathbf{n}$ are the curvature and principle unit normal
vector field of $\gamma $. Up to a change of coordinates, one can assumed as 
$\mathbf{u}=\left( 0,1\right) $ and $\gamma $ a graph locally given by $%
\gamma \left( s\right) =\left( s,f\left( s\right) \right) ,$ for $f:I\subset 
\mathbb{R}\rightarrow \mathbb{R}^{+}.$ Thereby, for $\alpha =1,$ Eq. %
\eqref{1.1} writes%
\begin{equation}
\frac{f^{\prime \prime }}{1+\left( f^{\prime }\right) ^{2}}=\frac{1}{f},
\label{1.2}
\end{equation}%
for each $s\in I.$ The solution of Eq. \eqref{1.2} is the catenary $f\left(
s\right) =\frac{1}{\lambda }\cosh \left( \lambda s+\mu \right) ,$ $\lambda
,\mu \in \mathbb{R},$ $\lambda \neq 0$, see \cite{Lopez10}.

As its two-dimensional analogue, Eq. \eqref{1.2} has a remarkable a physical point of view, which can be
initiated as follows: let the direction of gravity be choosen as $y-$axis.
Then Eq. \eqref{1.2} defines a configuration in which a uniform chain, whose
two ends are fixed and hanged under its own weight, is in balance with the
effect of the gravitational field. So, a catenary actually minimizes
potential energy under the influence of gravity force, in other words has
the lowest center of gravity (e.g. \cite{Gil}).

Let us now consider the smooth immersion $\sigma :M\rightarrow \mathbb{R}%
_{+}^{3}\left( \mathbf{u}\right) $ of an oriented surface $M$ in the
halfspace 
\begin{equation*}
\mathbb{R}_{+}^{3}\left( \mathbf{u}\right) =\left\{ p\in \mathbb{R}%
^{3},\left\langle p,\mathbf{u}\right\rangle >0\right\} ,
\end{equation*}%
for a fixed unit vector $\mathbf{u}\in \mathbb{R}^{3}$. Then, the
potential $\alpha -$energy of $\sigma $ in the direction of $\mathbf{u}$ is
defined by (\cite{Lopez7,Lopez8})
\begin{equation}
E\left( \varphi \right) =\int_{M}\left\langle \sigma \left( q\right)
,u\right\rangle ^{\alpha }dM,\text{ }q\in M,  \label{1.3}
\end{equation}%
where $dM$ refers to the measure on $M$ with respect to the induced metric
tensor from $\mathbb{R}^{3}$. If $\sigma $ is a critical point of Eq. \eqref{1.3}, it then
follows 
\begin{equation}
2H=\alpha \frac{\left\langle \xi ,\mathbf{u}\right\rangle }{\left\langle
\sigma ,\mathbf{u}\right\rangle },  \label{1.4}
\end{equation}%
where $H$ and $\xi $ are the mean curvature and unit normal vector field on $%
M.$ A surface in $\mathbb{R}^{3}$ fulfilling Eq. \eqref{1.4} is called 
\textit{singular minimal surface }or $\alpha $-\textit{minimal surface}
see \cite{Dierkes1}, \cite{Dierkes2}.

Eq. \eqref{1.4} is clearly an equation of mean curvature type and reduces to
the classical minimal surface equation when $\alpha =0$ \cite[p. 17]{Lopez4}%
. If we take $\mathbf{u}=(0,0,1)$ and $\alpha =1$ in Eq. \eqref{1.4}, then
the surface $M$ is said to be \textit{two-dimensional analogue of the
catenary} \cite{Bohme}.

L\'{o}pez \cite{Lopez7} proved that a singular minimal translation
surface in $\mathbb{R}^{3}$ with respect to a horizontal or a
vertical direction is a $\alpha -$\textit{catenary cylinder, }a generalized
cylinder (see \cite[p. 439]{Gray}) whose the base curve is a $\alpha -$catenary. This result was generalized to higher dimensions by the present authors \cite{EAM}.

Noting that the mean curvature $H$ in Eq. \eqref{1.4} is given with respect to the
Levi-Civita connection on $\mathbb{R}^{3}$, we modify it by considering the mean
curvatures arising via special semi-symmetric metric and non-metric
connections $\nabla $ and $D$ given by Eqs. \eqref{2.1} and \eqref{2.2}. 
We call the modified concepts $\nabla -$ and $D-$singular minimality and these allow us non-trivial and new problems. One of the problems is to find $\nabla -$ and $D-$singular minimal translation surfaces with respect to a horizontal direction and we solve it completely. 

\section{PRELIMINARIES}

Let $\left( \tilde{M},g\right) $ be a Riemannian manifold of dimension $3$
and $\tilde{\nabla}$ an affine connection on $\tilde{M}.$ Let us denote the
set of sections of a vector bundle $E\rightarrow \tilde{M}$ by $\Gamma
\left( E\right) $ and the set of tensor fields of type $\left( r,s\right) $
on $\tilde{M}$ by $\Gamma \left( T\tilde{M}^{\left( r,s\right) }\right) $.
Then the \textit{torsion tensor field} $T\in \Gamma \left( T\tilde{M}%
^{\left( 1,2\right) }\right) $ of $\tilde{\nabla}$ is defined by%
\begin{equation*}
T\left( X,Y\right) =\tilde{\nabla}_{X}Y-\tilde{\nabla}_{Y}X-\left[ X,Y\right]
,
\end{equation*}%
for $X,Y\in \Gamma \left( T\tilde{M}\right) $. Then the connection $\tilde{\nabla}
$ is called (see \cite{Murat},\cite{Yano2})

\begin{enumerate}
\item a (resp. \textit{non-}) \textit{symmetric connection } if $T$ (resp. does not)
vanishes identically;

\item a (resp. \textit{non-}) \textit{metric connection} if $g$ is (resp. not) parallel;

\item a \textit{semi-symmetric connection }if the following holds 
\begin{equation*}
T(X,Y)=\pi \left( Y\right) X-\pi \left( X\right) Y,
\end{equation*}%
for $\pi \left( X\right) =g\left( X,W\right) $, $\pi \in \Gamma \left( T%
\tilde{M}^{\left( 0,1\right) }\right) $, $W\in \Gamma \left( T\tilde{M}%
\right) $;

\item a \textit{Levi-Civita connection} if it is both symmetric
and metric.
\end{enumerate}

Let $\tilde{M}=\mathbb{R}^3$ and $\left\{ \partial _{x},\partial _{y},\partial _{z}\right\} $ the
standard basis on $\mathbb{R}^3$. Consider the certain semi-symmetric
metric and non-metric connections, respectively \cite{Wang}%
\begin{equation}
\nabla _{X}Y=\nabla _{X}^{L}Y+g\left( Y,\partial _{z}\right) X-g\left(
X,Y\right) \partial _{z}  \label{2.1}
\end{equation}%
and%
\begin{equation}
D_{X}Y=\nabla _{X}^{L}Y+g\left( Y,\partial _{z}\right) X,  \label{2.2}
\end{equation}%
where $\nabla ^{L}$ is the Levi-Civita connection on $\mathbb{R}^3$ and $X,Y\in
\Gamma \left( T\mathbb{R}^3\right) $. For Eqs. \eqref{2.1} and \eqref{2.2}, the
nonzero derivatives are given by%
\begin{equation*}
\nabla _{\partial _{x}}\partial _{x}=-\partial _{z},\text{ }\nabla
_{\partial _{x}}\partial _{z}=\partial _{x},\text{ }\nabla _{\partial
_{y}}\partial _{y}=-\partial _{z},\text{ }\nabla _{\partial _{y}}\partial
_{z}=\partial _{y},\text{ }
\end{equation*}%
and%
\begin{equation*}
D_{\partial _{x}}\partial _{z}=\partial _{x},\text{ }D_{\partial
_{y}}\partial _{z}=\partial _{y},\text{ }D_{\partial _{z}}\partial
_{z}=\partial _{z}.
\end{equation*}

Let $M$ be an oriented immersed surface into $\mathbb{R}^3.$ For any $X,Y\in
\Gamma \left( T\mathbb{R}^3\right) $ and $\xi \in \Gamma \left( T \mathbb{R}^{3^{\perp }}\right) $,
the \textit{Gauss formulae} with respect to $\nabla $ and $D$ follow 
\begin{equation*}
\nabla _{X}Y=\left( \nabla _{X}Y\right) ^{\intercal }+h^{\nabla }\left(
X,Y\right) \xi ,
\end{equation*}%
and 
\begin{equation*}
D_{X}Y=\left( \nabla _{X}Y\right) ^{\intercal }+h^{D}\left( X,Y\right) \xi ,
\end{equation*}%
where $\intercal $ is the projection on the tangent bundle of $M$ and $%
h^{\nabla }$ and $h^{D}$ are so-called the\textit{\ second fundamental forms}
with respect to $\nabla $ and $D,$ respectively. Let $\left\{
e_{1},e_{2}\right\} $ be an orthonormal tangent frame on $M$. Then the%
\textit{\ mean curvatures} of $M$ with respect to $\nabla $ and $D$ are
defined by 
\begin{equation*}
H^{\nabla }=\frac{1}{2}\left[ h^{\nabla }\left( e_{1},e_{1}\right)
+h^{\nabla }\left( e_{2},e_{2}\right) \right]
\end{equation*}%
and 
\begin{equation*}
H^{D}=\frac{1}{2}\left[ h^{D}\left( e_{1},e_{1}\right) +h^{D}\left(
e_{2},e_{2}\right) \right] .
\end{equation*}%
The surface $M$ is said to be \textit{minimal} with respect to $\nabla $ (resp. $D$) if $%
H^{\nabla }$ (resp. $H^{D}$) vanishes, respectively.

Let $g_{ij},$ $1\leq i,j\leq 2,$ denote the components of the induced metric
tensor on $M$ from the canonical metric. Then the mean
curvatures $H^{\nabla }$ and $H^{D}$\ are respectively given by%
\begin{equation}
H^{\nabla }=\frac{g_{22}h_{11}^{\nabla }-g_{12}\left( h_{12}^{\nabla
}+h_{21}^{\nabla }\right) +g_{11}h_{22}^{\nabla }}{2\det g_{ij}}
\label{2.3}
\end{equation}%
and 
\begin{equation}
H^{D}=\frac{g_{22}h_{11}^{D}-g_{12}\left( h_{12}^{D}+h_{21}^{D}\right)
+g_{11}h_{22}^{D}}{2\det g_{ij}},  \label{2.4}
\end{equation}%
where $h_{ij}^{\nabla }=  \left\langle \nabla _{f_{i}}f_{j},\xi \right\rangle  $ and $%
h_{ij}^{D}= \left\langle D _{f_{i}}f_{j},\xi \right\rangle ,$ for some basis $\left\{
f_{1},f_{2}\right\} $ of $\Gamma (TM).$

We utterly enable to introduce the notion of singular minimality with
respect to the connections $\nabla $ and $D$: let $\sigma :M\rightarrow 
\mathbb{R}^{3}$ be a smooth immersion of an oriented surface $M\ $and $%
\mathbf{u}\in \mathbb{R}^{3}$ a fixed unit vector. Let $H^{\nabla }$ and $%
H^{D}$ denote the mean curvatures with respect to $\nabla $ and $D$,
respectively. The surface $M$ is called $\nabla -$%
\textit{singular minimal surface} with respect to the vector $\mathbf{u}$\
if it holds%
\begin{equation}
2H^{\nabla }=\alpha \frac{\left\langle \xi ,\mathbf{u}\right\rangle }{%
\left\langle \sigma ,\mathbf{u}\right\rangle },  \label{2.5}
\end{equation}%
where $\xi $ is the unit normal vector field on $M$ and $\alpha$ a real constant. Accordingly the surface $M$
is called $D-$\textit{singular minimal surface} with respect to the vector $%
\mathbf{u}$ if it holds 
\begin{equation}
2H^{D}=\alpha \frac{\left\langle \xi ,\mathbf{u}\right\rangle }{\left\langle
\sigma ,\mathbf{u}\right\rangle }.  \label{2.6}
\end{equation}
It is obvious that these notions coincide with the usual minimality when $\alpha =0$ and so $\alpha \neq 0$ is assumed throughout the study.

\section{$\protect\nabla -$singular minimal translation surfaces}

In this section, we characterize $\nabla -$singular minimal translation surfaces in $%
\mathbb{R}^{3}$. As the translation property of the surface changes, Eq. \eqref{2.5} generates different tasks. Because there are three types of translation surfaces as follows
\begin{equation*}
z=f(x)+g(y), \text{ } y=f(x)+g(z), \text{ } x=f(y)+g(z), 
\end{equation*}
we state three separate results.

\begin{theorem}
A $\nabla -$singular minimal translation surface in $\mathbb{R}^{3}$ of type $z=f(x)+g(y)$ with respect to a horizontal vector $\mathbf{u}$ is a generalized cylinder and one of the following occurs

\begin{enumerate}
\item $f\left( x\right) =c_{1}$ and $g\left( y\right) =-\frac{1}{2}\ln
\left\vert \cos \left( 2y+c_{2}\right) \right\vert +c_{3},$

\item $g\left( y\right) =c_{4}+c_{5}y$ and $f$ is a solution of the ordinary
differential equation (ODE) 
\begin{equation}
f^{\prime \prime }=\frac{-\alpha }{\left( 1+c_{5}^{2}\right) x}\left(
f^{\prime }\right) ^{3}+\frac{2}{1+c_{5}^{2}}\left( f^{\prime }\right) ^{2}-%
\frac{\alpha }{x}f^{\prime }+2,  \label{3.1}
\end{equation}%
where $c_{1},...,c_{5}\in \mathbb{R}$ and $f^{\prime }=\frac{df}{dx}$, etc.
\end{enumerate}
\end{theorem}

\begin{proof}
The unit normal vector field
and mean curvature are computed by 
\begin{equation*}
\xi =\frac{-f^{\prime }\partial _{x}-g^{\prime }\partial _{y}+\partial _{z}}{%
\sqrt{1+\left( f^{\prime }\right) ^{2}+\left( g^{\prime }\right) ^{2}}}
\end{equation*}%
and 
\begin{equation*}
H^{\nabla}=\frac{\left[ 1+\left( g^{\prime }\right) ^{2}\right]f^{\prime \prime }
+\left[ 1+\left( f^{\prime }\right) ^{2}\right]g^{\prime \prime } -2\left[
1+\left( f^{\prime }\right) ^{2}+\left( g^{\prime }\right) ^{2}\right] }{2%
\left[ 1+\left( f^{\prime }\right) ^{2}+\left( g^{\prime }\right) ^{2}\right]
^{\frac{3}{2}}},
\end{equation*}%
where $g^{\prime }=\frac{dg}{dy}$ and so. Without of loss of generality, we may assume that $%
\mathbf{u}=\partial _{x}.$ Then Eq. \eqref{2.5} gives%
\begin{equation}
\frac{\left[ 1+\left( g^{\prime }\right) ^{2}\right]f^{\prime \prime }
+\left[ 1+\left( f^{\prime }\right) ^{2}\right]g^{\prime \prime } -2\left[
1+\left( f^{\prime }\right) ^{2}+\left( g^{\prime }\right) ^{2}\right] }{%
1+\left( f^{\prime }\right) ^{2}+\left( g^{\prime }\right) ^{2}}=-\alpha 
\frac{f^{\prime }}{x}.  \label{3.3}
\end{equation}
To solve Eq. \eqref{3.3}, we distinguish several cases: the first case is that $f(x)=f_{0}\in \mathbb{R}.$ Thus, Eq. \eqref{3.3} reduces to $g^{\prime \prime}=2\left( g^{\prime }\right) ^{2}+2.$ Solving this implies the
first statement of the theorem. The second case is that $f^{\prime }=f_{0}\neq 0$. Hence, Eq. \eqref{3.3} reduces to the following
polynomial equation of $x$%
\begin{equation*}
\left\{ g^{\prime \prime }\left( 1+f_{0}^{2}\right) -2\left[
1+f_{0}^{2}+\left( g^{\prime }\right) ^{2}\right] \right\} x+\alpha f_{0}%
\left[ 1+f_{0}^{2}+\left( g^{\prime }\right) ^{2}\right] =0,
\end{equation*}%
which means that this case is false because the constant term of the
polynomial equation cannot be zero. The last case is that $f^{\prime \prime }\neq 0$. Then taking partial derivative Eq. \eqref{3.3}
with respect to $y$ leads to
\begin{equation}
2\left( -2+f^{\prime \prime }\right) g^{\prime }g^{\prime \prime }+\left[
1+\left( f^{\prime }\right) ^{2}\right] g^{\prime \prime \prime }=-2\alpha
g^{\prime }g^{\prime \prime }\frac{f^{\prime }}{x}.  \label{3.4}
\end{equation}%
That $g^{\prime }=g_{0},$ $g_{0}\in \mathbb{R},$ is clearly a solution of
Eq. \eqref{3.4}. Then Eq. \eqref{3.3} reduces to%
\begin{equation*}
\left[ -2-2g_{0}^{2}+\left( 1+g_{0}^{2}\right) f^{\prime \prime }-2\left(
f^{\prime }\right) ^{2}\right]x +\alpha f^{\prime }\left[ 1+g_{0}^{2}+\left(
f^{\prime }\right) ^{2}\right] =0,
\end{equation*}%
or%
\begin{equation*}
f^{\prime \prime }+\frac{\alpha }{\left( 1+g_{0}^{2}\right) x}\left(
f^{\prime }\right) ^{3}-\frac{2}{1+g_{0}^{2}}\left( f^{\prime }\right) ^{2}+%
\frac{\alpha }{x}f^{\prime }-2=0,
\end{equation*}%
which gives the second statement of the theorem. Then the proof finishes if we show that Eq. \eqref{3.4} has no solution for $f^{\prime \prime }g^{\prime \prime }\neq 0.$ By contradiction, assume that $f^{\prime \prime }g^{\prime \prime }\neq 0.$  Dividing Eq. \eqref{3.4} with $%
2g^{\prime }g^{\prime \prime },$ we get%
\begin{equation}
\left[ 1+\left( f^{\prime }\right) ^{2}\right] \frac{g^{\prime \prime \prime
}}{2g^{\prime }g^{\prime \prime }}+f^{\prime \prime }+\alpha \frac{f^{\prime
}}{x}-2=0,  \label{3.5}
\end{equation}%
yielding%
\begin{equation}
g^{\prime \prime \prime }=2cg^{\prime }g^{\prime \prime },\text{ }c\in 
\mathbb{R}.  \label{3.6}
\end{equation}%
Integrating Eq. \eqref{3.6} gives 
\begin{equation}
g^{\prime \prime }=c\left( g^{\prime }\right) ^{2}+d,\text{ }d\in \mathbb{R}.
\label{3.7}
\end{equation}%
Substituting Eq. \eqref{3.7} to Eq. \eqref{3.3} leads to%
\begin{equation}
f^{\prime \prime }+\left[ 1+\left( f^{\prime }\right) ^{2}\right] \left(
-2+d+\alpha \frac{f^{\prime }}{x}\right) =0.  \label{3.8}
\end{equation}%
From Eqs \eqref{3.5} and \eqref{3.8}, we conclude two equations%
\begin{equation}
f^{\prime \prime }=\frac{\left( -2+\alpha \frac{f^{\prime }}{x}\right)
\left( -2+d+\alpha \frac{f^{\prime }}{x}\right) }{2+c-d-\alpha \frac{%
f^{\prime }}{x}}  \label{3.9}
\end{equation}%
and 
\begin{equation}
\left( -2-c+d\right) \left( f^{\prime }\right) ^{2}+\alpha \frac{\left(
f^{\prime }\right) ^{3}}{x}+d-c=0.  \label{3.10}
\end{equation}%
Notice that the denominator in Eq. \eqref{3.9} is not zero because the contradiction $f^{\prime \prime}=0$ is obtained otherwise. Taking derivative of Eq. \eqref{3.10} leads to%
\begin{equation}
\left[ 2\left( -2-c+d\right) +\frac{3\alpha f^{\prime }}{x}\right] f^{\prime
\prime }=\alpha \left( \frac{f^{\prime }}{x}\right) ^{2}.  \label{3.11}
\end{equation}%
Considering Eq. \eqref{3.9} to Eq. \eqref{3.11} gives a polynomial equation
of $\frac{f^{\prime }}{x}$ 
\begin{equation}  \label{3.12}
\left. 
\begin{array}{c}
\alpha ^{2}\left( 1+3\alpha \right) \left( \frac{f^{\prime }}{x}\right)
^{3}+\alpha \left[ \alpha \left( -16-2c+5d\right) -2-c+d\right] \left( \frac{%
f^{\prime }}{x}\right) ^{2} \\ 
2\alpha \left[ \left( -4+d\right) ( -2-c+d)-3( -2+d) %
\right] \frac{f^{\prime }}{x}-4\left( -2-c+d\right) \left( -2+d\right) =0,%
\end{array}%
\right.
\end{equation}%
in which the fact that the leading coefficient vanish yields $\alpha =\frac{%
-1}{3}$ and therefore Eq. \eqref{3.12} reduces%
\begin{equation}  \label{3.13}
	\left. 
	\begin{array}{c}
\frac{-10+c+2d}{9}\left( \frac{f^{\prime }}{x}\right) ^{2}-\frac{2}{3}\left[
\left( -4+d\right) \left( -2-c+d\right) -3\left( -2+d\right) \right] \frac{%
f^{\prime }}{x}\\
-4\left( -2-c+d\right) \left( -2+d\right) =0.
\end{array}%
\right.
\end{equation}%
The constant term must be zero and for this reason assume that $d=2.$
Then Eq. \eqref{3.13} gives the contradiction $c=6$\ and $c=0.$ This means that $d\neq 2$ and $-2-c+d=0$ and therefore the
coefficient of the term of degree 1 cannot vanish, a contradiction.
\end{proof}

The ODE \eqref{3.1} admits a reduction of order%
\begin{equation}
u^{\prime }=\frac{-\alpha }{\left( 1+c_{5}^{2}\right) x}u^{3}+\frac{2}{%
1+c_{5}^{2}}u^{2}-\frac{\alpha }{x}u+2,  \label{3.2}
\end{equation}%
where $u=u\left( x\right) =f^{\prime }$ and $u^{\prime }\left( x\right)
=f^{\prime \prime }.$ Due to $\alpha \neq 0,$ the ODE (3.13) is an Abel
equation of the first kind given by%
\begin{equation*}
u^{\prime }=p_{3}\left( x\right) u^{3}+p_{2}\left( x\right)
u^{2}+p_{1}\left( x\right)u+p_{0}\left( x\right)
\end{equation*}
and no admits a general solution in terms of known functions, excepting very
special cases depending on the functions $p_{0}\left( x\right)
,...,p_{3}\left( x\right) .$ We refer to \cite{Pana, Polyanin} for more
details. Notice also that the surface given by the first statement of Theorem 1 is both $\nabla -$singular minimal and $\nabla -$minimal.

\begin{theorem}
A $\nabla -$singular minimal translation surface in $\mathbb{R}^{3}$ of type $y=f(x)+g(z)$ with respect to a horizontal vector $\mathbf{u}$ is either a plane parallel to $\mathbf{u}$ or a generalized
cylinder satisfying one of the following

\begin{enumerate}
\item $f(x)=c_{1}$ and 
\begin{equation*}
g\left( z\right) =\pm \frac{1}{2}\arctan \left( \frac{1}{c_{2}}\sqrt{%
e^{4z}-c_{2}^{2}}\right) +c_{3},
\end{equation*}

\item $g\left( z\right) =c_{4}+c_{5}z$ and $f$ is a solution of the ODE 
\begin{equation*}
f^{\prime \prime }=\frac{-\alpha }{\left( 1+c_{5}^{2}\right) x}\left(
f^{\prime }\right) ^{3}-\frac{2c_{5}}{1+c_{5}^{2}}\left( f^{\prime }\right)
^{2}-\frac{\alpha }{x}f^{\prime }-2c_{5},
\end{equation*}%
where $c_{1},...,c_{5}\in \mathbb{R}$, $c_{2}\neq 0$, and $f^{\prime }=\frac{df}{dx}$, etc.
\end{enumerate}
\end{theorem}

\begin{proof}
The unit normal vector field
and mean curvature are given by 
\begin{equation*}
\xi =\frac{f^{\prime }\partial _{x}-\partial _{y}+g^{\prime }\partial _{z}}{%
\sqrt{1+\left( f^{\prime }\right) ^{2}+\left( g^{\prime }\right) ^{2}}}
\end{equation*}%
and 
\begin{equation*}
2H^{\nabla}=-\frac{\left[ 1+\left( g^{\prime }\right) ^{2}\right] f^{\prime \prime }+%
\left[ 1+\left( f^{\prime }\right) ^{2}\right] g^{\prime \prime }+2\left[
1+\left( f^{\prime }\right) ^{2}+\left( g^{\prime }\right) ^{2}\right]
g^{\prime }}{\left[ 1+\left( f^{\prime }\right) ^{2}+\left( g^{\prime
}\right) ^{2}\right] ^{\frac{3}{2}}},
\end{equation*}%
where $g^{\prime }=\frac{dg}{dz}$ and so. Without of loss of generality we may assume that $\mathbf{u}=\partial _{x}$.
Therefore, Eq. \eqref{2.5} implies%
\begin{equation}
\frac{f^{\prime \prime }\left[ 1+\left( g^{\prime }\right) ^{2}\right]
+g^{\prime \prime }\left[ 1+\left( f^{\prime }\right) ^{2}\right] +2\left[
1+\left( f^{\prime }\right) ^{2}+\left( g^{\prime }\right) ^{2}\right]
g^{\prime }}{1+\left( f^{\prime }\right) ^{2}+\left( g^{\prime }\right) ^{2}}%
=-\alpha \frac{f^{\prime }}{x}.  \label{3.14}
\end{equation}%
To solve Eq. \eqref{3.14}, we distinguish several cases: the first case is that $f(x)=f_{0}\in \mathbb{R}$. Then, Eq. \eqref{3.14} reduces to $g^{\prime
\prime }=-2\left[ 1+\left( g^{\prime }\right) ^{2}\right] g^{\prime }.$ That $g^{\prime }=0$ is a trivial solution of this equation and leads the surface to a plane parallel to $\mathbf{u}$. Otherwise, $g^{\prime } \neq 0$, after solving it, we obtain the first statement of the theorem. The second case is that $f^{\prime }=f_{0}\neq 0.$ Then Eq. \eqref{3.14} reduces to a
polynomial equation of $x$ in the form 
\begin{equation*}
\left\{ \left[ 1+f_{0}^{2}\right] g^{\prime \prime }+2\left[
1+f_{0}^{2}+\left( g^{\prime }\right) ^{2}\right] g^{\prime }\right\}
x+\alpha f_{0}\left( 1+f_{0}^{2}+\left( g^{\prime }\right) ^{2}\right) =0,
\end{equation*}%
which implies that this case is false because the constant term of the
polynomial equation cannot be zero. The last case is that $f^{\prime \prime }\neq 0$. Taking partial derivative Eq. \eqref{3.14}
with respect to $z$, we get%
\begin{equation}
6\left( g^{\prime }\right) ^{2}g^{\prime \prime }+2\left( f^{\prime \prime
}+\alpha \frac{f^{\prime }}{x}\right) g^{\prime }g^{\prime \prime }+\left[
1+\left( f^{\prime }\right) ^{2}\right] \left( 2g^{\prime \prime }+g^{\prime
\prime \prime }\right) =0.  \label{3.15}
\end{equation}%
That $g^{\prime }=g_{0}\in \mathbb{R}$ is clearly a solution to Eq. %
\eqref{3.15}. So, Eq. \eqref{3.14} reduces to%
\begin{equation*}
f^{\prime \prime }+\frac{\alpha }{\left( 1+g_{0}^{2}\right) x}\left(
f^{\prime }\right) ^{3}+\frac{2g_{0}}{1+g_{0}^{2}}\left( f^{\prime }\right)
^{2}+\frac{\alpha }{x}f^{\prime }+2g_{0}=0,
\end{equation*}%
which gives the second statement of the theorem. Then the proof finishes if we show that Eq. \eqref{3.15} has no solution for $f^{\prime \prime }g^{\prime \prime }\neq 0.$ By contradiction, assume that $f^{\prime \prime }g^{\prime \prime }\neq 0.$  Dividing Eq. \eqref{3.15} with $%
2g^{\prime }g^{\prime \prime }$, we have 
\begin{equation}
\left( \frac{g^{\prime \prime \prime }}{2g^{\prime }g^{\prime \prime }}+%
\frac{1}{g^{\prime }}\right) \left[ 1+\left( f^{\prime }\right) ^{2}\right]
+ f^{\prime \prime }+\alpha \frac{f^{\prime }}{x}+3g^{\prime
}=0.  \label{3.16}
\end{equation}%
Taking partial derivative of Eq. \eqref{3.16} with respect to $x$ and $z$
leads to%
\begin{equation*}
f^{\prime }f^{\prime \prime }\left( \frac{g^{\prime \prime \prime }}{%
2g^{\prime }g^{\prime \prime }}+\frac{1}{g^{\prime }}\right) ^{\prime }=0
\end{equation*}%
or, owing to $f^{\prime }f^{\prime \prime }\neq 0,$ the term $
g^{\prime \prime \prime }/\left( 2g^{\prime }g^{\prime \prime }\right)
+1/g^{\prime }$ becomes a constant. Therefore, the partial
derivative of Eq. \eqref{3.16} with respect to $z$ gives $g^{\prime \prime }=0$, a contradiction
\end{proof}

 Note that, as in previous result, the surface given by the first statement of Theorem 2 is both $\nabla -$singular minimal and $\nabla -$minimal.
\begin{theorem}
A $\nabla -$singular minimal translation surface in $\mathbb{R}^{3}$ of type $x=f(y)+g(z)$ with respect to a horizontal vector $\mathbf{u}$ is a generalized cylinder and one of the following occurs

\begin{enumerate}
\item $f\left( y\right) =c_{1}$ and $g$ is a solution of the autonomous ODE%
\begin{equation*}
g^{\prime \prime }=\left( \frac{\alpha }{c_{1}+g}-2g^{\prime }\right) \left[
1+\left( g^{\prime }\right) ^{2}\right] ,\text{ }g^{\prime \prime }\neq 0;
\end{equation*}

\item $g\left( z\right) =c_{2}$ and 
\begin{equation*}
y=\pm \int \left[ c_{3}\left( f+c_{2}\right) ^{2\alpha }+c_{4}\right]
^{-1/2}df,
\end{equation*}%
for $c_{1},...,c_{4}\in \mathbb{R},$ $c_{3}\neq 0,$ and $g^{\prime }=\frac{dg}{dz},$ etc.
\end{enumerate}
\end{theorem}

\begin{proof}
The unit normal vector field and
mean curvature are 
\begin{equation*}
\xi =\frac{\left( \partial _{x}-f^{\prime }\partial _{y}-g^{\prime }\partial
_{z}\right) }{\sqrt{1+\left( f^{\prime }\right) ^{2}+\left( g^{\prime
}\right) ^{2}}}
\end{equation*}%
and 
\begin{equation*}
2H^{\nabla}=\frac{\left[ 1+\left( g^{\prime }\right) ^{2}\right] f^{\prime \prime }+%
\left[ 1+\left( f^{\prime }\right) ^{2}\right] g^{\prime \prime }+2\left[
1+\left( f^{\prime }\right) ^{2}+\left( g^{\prime }\right) ^{2}\right]
g^{\prime }}{\left[ 1+\left( f^{\prime }\right) ^{2}+\left( g^{\prime
}\right) ^{2}\right] ^{\frac{3}{2}}},
\end{equation*}%
where $f^{\prime }=\frac{df}{dy}$ and so. Without of loss of generality we may assume that $\mathbf{u}=\partial _{x}$.
Hence Eq. \eqref{2.5} turns to 
\begin{equation}
\frac{\left[ 1+\left( g^{\prime }\right) ^{2}\right] f^{\prime \prime }+%
\left[ 1+\left( f^{\prime }\right) ^{2}\right] g^{\prime \prime }+2\left[
1+\left( f^{\prime }\right) ^{2}+\left( g^{\prime }\right) ^{2}\right]
g^{\prime }}{1+\left( f^{\prime }\right) ^{2}+\left( g^{\prime }\right) ^{2}}%
=\alpha \frac{1}{f+g},  \label{3.17}
\end{equation}%
in which both $f$ and $g$ cannot be constant simultaneously because the
situation $\alpha =0$ is obtained otherwise. We distinguish the remaining
cases: the first case is that $f(y)=f_{0}\in \mathbb{R}$. Eq. \eqref{3.17} writes%
\begin{equation}
g^{\prime \prime }+2\left[ 1+\left( g^{\prime }\right) ^{2}\right] g^{\prime
}=\alpha \frac{1+\left( g^{\prime }\right) ^{2}}{f_{0}+g}.  \label{3.18}
\end{equation}%
in which $g$ must be non-linear because the situation $\alpha =0$ is
obtained otherwise. This concludes the first statement of the theorem. The second case is that $f\left( y\right) =d+cy,$ $c,d\in \mathbb{R},$ $c\neq 0.$ Then Eq. \eqref{3.17} turns to 
\begin{equation}
\frac{\left( 1+c^{2}\right) g^{\prime \prime }}{1+c^{2}+\left( g^{\prime
}\right) ^{2}}+2g^{\prime }=\alpha \frac{1}{d+cy+g}.  \label{3.19}
\end{equation}%
By taking partial derivative of Eq. \eqref{3.19} with respect to $y,$ we
obtain the contradiction $\alpha c=0.$ The third case is that $f^{\prime \prime }\neq 0$ and $g=g_{0}\in \mathbb{R}$. Then Eq. %
\eqref{3.17} reduces to 
\begin{equation}
\frac{f^{\prime \prime }}{1+\left( f^{\prime }\right) ^{2}}=\alpha \frac{1}{%
f+g_{0}}.  \label{3.20}
\end{equation}%
By multiplying Eq. \eqref{3.20} with $2f^{\prime }$ and taking first
integral, we obtain%
\begin{equation}
f^{\prime }=\pm \sqrt{c\left( f+g_{0}\right) ^{2\alpha }-1},\text{ }c\in 
\mathbb{R},\text{ }c\neq 0.  \label{3.21}
\end{equation}%
Taking derivative of Eq. \eqref{3.21}, we can conclude%
\begin{equation*}
f^{\prime \prime }=\alpha c\left( f+g_{0}\right) ^{2\alpha -1},
\end{equation*}%
which is known as \textit{Emden--Fowler} equation (see \cite{Polyanin})
and the solution follows%
\begin{equation*}
y=\pm \int \left[ c\left( f+g_{0}\right) ^{2\alpha }+d\right] ^{-1/2}df+e,%
\text{ }d,e\in \mathbb{R},
\end{equation*}%
which proves the second statement of the theorem. The fourth case is that $f^{\prime \prime} \neq 0$ and $g(z)=d+cz,$ $c,d\in \mathbb{R},$ $c\neq0. $ Then Eq. \eqref{3.17} reduces to%
\begin{equation}
\frac{\left[ 1+c^{2}\right] f^{\prime \prime }}{1+c^{2}+\left( f^{\prime
}\right) ^{2}}+2c=\alpha \frac{1}{d+cz+f}.  \label{3.22}
\end{equation}%
By taking partial derivative of Eq. \eqref{3.22} with respect to $z,$ we
obtain the contradiction $\alpha c=0.$ Therefore the proof finishes if we show that Eq. \eqref{3.17} has no solution for $f^{\prime \prime }g^{\prime \prime}\neq 0$. By contradiction, assume that $f^{\prime \prime }g^{\prime \prime}\neq 0$. Then Eq. \eqref{3.17}
can be rewritten as%
\begin{equation}  \label{3.23}
\begin{array}{c}
\left. \left( f+g\right) \left\{ \left[ 1+\left( g^{\prime }\right) ^{2}%
\right] f^{\prime \prime }+\left[ 1+\left( f^{\prime }\right) ^{2}\right]
g^{\prime \prime }+2\left[ 1+\left( f^{\prime }\right) ^{2}+\left( g^{\prime
}\right) ^{2}\right] g^{\prime }\right\} \right. \\ 
=\alpha \left[ 1+\left( f^{\prime }\right) ^{2}+\left( g^{\prime }\right)
^{2}\right] .%
\end{array}%
\end{equation}%
If we take partial derivative of Eq. \eqref{3.23} with respect to $y$
and $z$ and afterwards divide the generated equation with $f^{\prime
}f^{\prime \prime }g^{\prime }g^{\prime \prime },$ we can deduce%
\begin{equation}
\begin{array}{c}
\left. 4+\right. \frac{f^{\prime \prime \prime }}{f^{\prime }f^{\prime
\prime }}\left[ \frac{1+\left( g^{\prime }\right) ^{2}}{g^{\prime \prime }}%
+2g\right] +\left( \frac{g^{\prime \prime \prime }}{g^{\prime }g^{\prime
\prime }}+\frac{2}{g^{\prime }}\right) \left[ \frac{1+\left( f^{\prime
}\right) ^{2}}{f^{\prime \prime }}+2f\right] \\ 
+\frac{2ff^{\prime \prime \prime }}{f^{\prime }f^{\prime \prime }}+\frac{%
2gg^{\prime \prime \prime }}{g^{\prime }g^{\prime \prime }}+\frac{4g}{%
g^{\prime }}+\frac{4g^{\prime }}{g^{\prime \prime }}+\frac{6g^{\prime }}{%
f^{\prime \prime }}=0.%
\end{array}
\label{3.24}
\end{equation}%
The partial derivative of Eq. \eqref{3.24} with respect to $y$ and $z$ leads
to%
\begin{equation}
\left( \frac{f^{\prime \prime \prime }}{f^{\prime }f^{\prime \prime }}%
\right) ^{\prime }\left[ \frac{1+\left( g^{\prime }\right) ^{2}}{g^{\prime
\prime }}+2g\right] ^{\prime }+\left( \frac{g^{\prime \prime \prime }}{%
g^{\prime }g^{\prime \prime }}+\frac{2}{g^{\prime }}\right) ^{\prime }\left[ 
\frac{1+\left( f^{\prime }\right) ^{2}}{f^{\prime \prime }}+2f\right]
^{\prime }-\frac{6g^{\prime \prime }f^{\prime \prime \prime }}{\left(
f^{\prime \prime }\right) ^{2}}=0.  \label{3.25}
\end{equation}%
We have subcases:

\begin{enumerate}
\item $f^{\prime \prime }=f_{0}\in \mathbb{R},$ $f_{0}\neq 0.$ Then the
partial derivative of Eq. \eqref{3.24} with respect to $y$ leads to 
\begin{equation}
\frac{g^{\prime \prime \prime }}{g^{\prime }g^{\prime \prime }}+\frac{2}{%
g^{\prime }}=0.  \label{3.26}
\end{equation}%
Therefore Eq. \eqref{3.24} reduces to%
\begin{equation}
2+\frac{2g^{\prime }}{g^{\prime \prime }}+%
\frac{3g^{\prime }}{f_{0}}=0.  \label{3.27}
\end{equation}%
On the other hand, a first integration of Eq. \eqref{3.26} gives%
\begin{equation}
g^{\prime \prime }=-2g^{\prime }+c,\text{ }c\in \mathbb{R}.  \label{3.28}
\end{equation}%
By substituting Eq. \eqref{3.28} into Eq. \eqref{3.27}, we obtain a
polynomial equation of $g^{\prime }$ whose the leading coefficient coming
from the term $\left( g^{\prime }\right) ^{2}$ is $\frac{-6}{f_{0}},$ a contradiction.

\item $f^{\prime \prime \prime }=2cf^{\prime }f^{\prime \prime },$ $c\in 
\mathbb{R},$ $c\neq 0.$ A first integration yields $f^{\prime \prime
}=c\left( f^{\prime }\right) ^{2}+d,$ $d\in \mathbb{R}.$ Then Eq. %
\eqref{3.25} reduces to%
\begin{equation*}
\left( \frac{g^{\prime \prime \prime }}{g^{\prime }g^{\prime \prime }}+\frac{%
2}{g^{\prime }}\right) ^{\prime }\left[ \frac{2d}{c}-1+\left( f^{\prime
}\right) ^{2}\right] -6g^{\prime \prime }=0,
\end{equation*}%
which implies $\left( \frac{g^{\prime \prime \prime }}{g^{\prime }g^{\prime
\prime }}+\frac{2}{g^{\prime }}\right) ^{\prime }=0.$ This gives from Eq. %
\eqref{3.25} the contradiction $g^{\prime \prime }f^{\prime \prime \prime
}=0.$

\item $\left( f^{\prime \prime \prime }/f^{\prime }f^{\prime \prime }\right)
^{\prime }\neq 0.$ Eq. \eqref{3.25} can be rewritten by dividing $g^{\prime\prime}\left(
f^{\prime \prime \prime }/f^{\prime }f^{\prime \prime }\right) ^{\prime }$ as%
\begin{equation}
A\left( y\right) B\left( z\right) =C\left(y\right) +D\left( z\right) ,
\label{3.29}
\end{equation}%
where%
\begin{equation*}
A\left( y\right) =\frac{\left[ \left\{ 1+\left( f^{\prime }\right)
^{2}\right\} /f^{\prime \prime }+2f\right] ^{\prime }}{\left( f^{\prime
\prime \prime }/f^{\prime }f^{\prime \prime }\right) ^{\prime }},\text{ }%
B\left( z\right) =\left( g^{\prime \prime \prime }/g^{\prime }g^{\prime
\prime }+2/g^{\prime }\right) ^{\prime }/g^{\prime \prime }
\end{equation*}%
and%
\begin{equation*}
C\left(y\right) =6\frac{f^{\prime \prime \prime }/\left( f^{\prime \prime
}\right) ^{2}}{\left( f^{\prime \prime \prime }/f^{\prime }f^{\prime \prime
}\right) ^{\prime }},\text{ }D\left( z\right) =-\left[ \left\{ 1+\left(
g^{\prime }\right) ^{2}\right\} /g^{\prime \prime }+2g\right] ^{\prime
}/g^{\prime \prime }.
\end{equation*}%
The functions $A,B,C,D$ from Eq. \eqref{3.29} must be all constant. Let us
put $B\left( z\right) =B_{0}$ and $D\left( z\right) =-D_{0},$ $B_{0},D_{0}\in 
\mathbb{R}.$ Therefore, we get%
\begin{equation}
\left[ \frac{g^{\prime \prime \prime }}{g^{\prime }g^{\prime \prime }}+\frac{%
2}{g^{\prime }}\right] ^{\prime }=B_{0}g^{\prime \prime }  \label{3.30}
\end{equation}%
and 
\begin{equation}
4g^{\prime }-\left[ 1+\left( g^{\prime }\right) ^{2}\right] \frac{g^{\prime
\prime \prime }}{\left( g^{\prime \prime }\right) ^{2}}=D_{0}g^{\prime
\prime }  \label{3.31}
\end{equation}%
A first integration of Eq. \eqref{3.30} yields%
\begin{equation}
\frac{g^{\prime \prime \prime }}{g^{\prime }g^{\prime \prime }}+\frac{2}{%
g^{\prime }}=B_{0}g^{\prime }+d_{1}.  \label{3.32}
\end{equation}%
for $d_{1}\in \mathbb{R}.$ Multiplying Eq. \eqref{3.32} with $g^{\prime
}g^{\prime \prime }$ and then taking first integration the generated
equation gives%
\begin{equation}
g^{\prime \prime }=\frac{B_{0}}{3}\left( g^{\prime }\right) ^{3}+\frac{d_{1}%
}{2}\left( g^{\prime }\right) ^{2}-2g^{\prime }+d_{2},  \label{3.33}
\end{equation}%
for $d_{2}\in \mathbb{R}.$ By considering Eq. \eqref{3.33} into Eq. %
\eqref{3.32}, we deduce that 
\begin{equation}
g^{\prime \prime \prime }=\frac{B_{0}^2}{3}\left( g^{\prime }\right) ^{5}+%
\text{rest terms.}  \label{3.34}
\end{equation}%
Substituting Eqs. \eqref{3.33} and \eqref{3.34} into Eq. \eqref{3.31}, we
have a polynomial equation of $g^{\prime }$ whose the coefficient coming from $\left( g^{\prime }\right) ^{7}$ is $\frac{%
B_{0}^2}{9} ,$ yielding $B_{0}=0.$ 
It follows from Eq. \eqref{3.32}%
\[
g^{\prime \prime \prime }=\left( -2+d_{1}g^{\prime }\right) g^{\prime \prime
}
\]%
and plugging it into Eq. \eqref{3.31}%
\begin{equation}
4g^{\prime }g^{\prime \prime }-\left[ 1+\left( g^{\prime }\right) ^{2}\right]
\left( -2+d_{1}g^{\prime }\right) =D_{0}\left( g^{\prime \prime }\right)
^{2}.  \label{3.35}
\end{equation}%
Considering Eq. \eqref{3.33} into Eq. \eqref{3.35} leads to a polynomial
equation of $g^{\prime }$ whose the coefficient coming from $\left(
g^{\prime }\right) ^{3}$ is $d_{1}$ which must vanish. Therefore Eqs. %
\eqref{3.33} and \eqref{3.35} reduce to%
\[
g^{\prime \prime }=-2g^{\prime }+d_{2}
\]%
and%
\[
4g^{\prime }g^{\prime \prime }+2\left( g^{\prime }\right) ^{2}+2=D_{0}\left(
g^{\prime \prime }\right) ^{2},
\]%
respectively. From these two equations, we can conclude a polynomial
equation of $g^{\prime }$ 
\[
\left( 4D_{0}+6\right) \left( g^{\prime }\right) ^{2}-4\left(
D_{0}+d_{2}\right) g^{\prime }+D_{0}d_{2}^{2}-2=0,
\]%
which gives a contradiction because the constant term cannot vanish.
\end{enumerate}
\end{proof}

\section{$D-$singular minimal translation surfaces}

As in previous section, we characterize translation surfaces in $%
\mathbb{R}^{3}$ of each type to be $D-$singular minimal through the following results.

\begin{theorem}
A $D-$singular minimal translation surface in $\mathbb{R}^{3}$ of type $z=f(x)+g(y)$
with respect to a horizontal vector $\mathbf{u}$ is either a plane parallel
to $\mathbf{u}$ or a generalized cylinder satisfying $g(y)=c_{1}+c_{2}y$ and%
\begin{equation*}
f\left( x\right) =\pm \left\vert c_{3}\right\vert \sqrt{1+c_{2}^{2}}\int
\left( x^{2\alpha }-c_{3}^{2}\right) ^{-1/2}dx,
\end{equation*}%
where $c_{1},c_{2},c_{3}\in \mathbb{R}$, $c_{3}\neq 0.$
\end{theorem}

\begin{proof}
The unit normal vector field and mean curvature follow%
\begin{equation*}
\xi =\frac{-f^{\prime }\partial _{x}-g^{\prime }\partial _{y}+\partial _{z}}{%
\sqrt{1+\left( f^{\prime }\right) ^{2}+\left( g^{\prime }\right) ^{2}}}
\end{equation*}%
and 
\begin{equation*}
H^{D}=\frac{f^{\prime \prime }\left[ 1+\left( g^{\prime }\right) ^{2}\right]
+g^{\prime \prime }\left[ 1+\left( f^{\prime }\right) ^{2}\right] }{2\left[
1+\left( f^{\prime }\right) ^{2}+\left( g^{\prime }\right) ^{2}\right] ^{%
\frac{3}{2}}},
\end{equation*}%
where $f^{\prime }=\frac{df}{dx},g^{\prime }=\frac{dg}{dy}$ and so. Without of loss of generality we may assume that $\mathbf{u}=\partial _{x}$.
Hence Eq. \eqref{2.6} turns to 
\begin{equation}
\frac{\left[ 1+\left( g^{\prime }\right) ^{2}\right] f^{\prime \prime }+%
\left[ 1+\left( f^{\prime }\right) ^{2}\right] g^{\prime \prime }}{1+\left(
f^{\prime }\right) ^{2}+\left( g^{\prime }\right) ^{2}}=-\alpha \frac{%
f^{\prime }}{x}.  \label{4.1}
\end{equation}%
Eq. \eqref{4.1} has no solution if $g^{\prime \prime }\neq 0,$ see \cite[%
Theorem 5]{Lopez7}. Therefore, we have $g^{\prime }=g_{0}\in \mathbb{R}$ and
Eq. \eqref{4.1} reduces to%
\begin{equation}
\frac{\left( 1+g_{0}^{2}\right) f^{\prime \prime }}{f^{\prime }\left[
1+g_{0}^{2}+\left( f^{\prime }\right) ^{2}\right] }=- \frac{\alpha}{x},
\label{4.2}
\end{equation}%
where being $f$ a constant is a trivial solution, implying $M$
is a plane parallel to the vector $\mathbf{u}.$ Assume that $f$ is no constant and then we can easily infer from Eq. %
\eqref{4.2} that $f$ cannot be a linear function, i.e. $f^{\prime \prime
}=0, $ due to $\alpha \neq 0.$ Hence, a first integration of Eq. \eqref{4.2}
yields%
\begin{equation*}
\frac{f^{\prime }}{\sqrt{1+g_{0}^{2}+\left( f^{\prime }\right) ^{2}}}%
=cx^{-\alpha },\text{ }c\in \mathbb{R},\text{ }c\neq 0,
\end{equation*}%
or%
\begin{equation}
f^{\prime }=\pm \left\vert c\right\vert \frac{\sqrt{1+g_{0}^{2}}}{\sqrt{%
x^{2\alpha }-c^{2}}}.  \label{4.3}
\end{equation}%
A first integration of Eq. \eqref{4.3} completes the proof.
\end{proof}

For a translation surface of type $y=f(x)+g(z)$, the unit normal vector field and mean
curvature follow%
\begin{equation*}
\xi =\frac{f^{\prime }\partial _{x}-\partial _{y}+g^{\prime }\partial _{z}}{%
\sqrt{1+\left( f^{\prime }\right) ^{2}+\left( g^{\prime }\right) ^{2}}}
\end{equation*}%
and 
\begin{equation*}
H^{D}=\frac{f^{\prime \prime }\left[ 1+\left( g^{\prime }\right) ^{2}\right]
+g^{\prime \prime }\left[ 1+\left( f^{\prime }\right) ^{2}\right] }{2\left[
1+\left( f^{\prime }\right) ^{2}+\left( g^{\prime }\right) ^{2}\right] ^{%
\frac{3}{2}}},
\end{equation*}%
where $f^{\prime }=\frac{df}{dx}, g^{\prime }=\frac{dg}{dz}$ and so. Then, with respect to the horizontal vector $\mathbf{u}=\partial _{x},$ the $%
D-$singular minimality equation is similar to Eq. \eqref{4.1}\ up to a sign. Thereby, for such a surface, we can state a similar result to Theorem 4 without proof by replacing $\alpha$ with $-\alpha$.

\begin{theorem}
A $D-$singular minimal translation surface in $\mathbb{R}^{3}$ of type $y=f(x)+g(z)$
with respect to a horizontal vector $\mathbf{u}$ is either a plane parallel
to $\mathbf{u}$ or a generalized cylinder satisfying $g(z)=c_{1}+c_{2}z$ and%
\begin{equation*}
f\left( x\right) =\pm \left\vert c_{3}\right\vert \sqrt{1+c_{2}^{2}}\int
x^{\alpha }\left( 1-c_{3}^{2}x^{2\alpha }\right) ^{-1/2}dx,
\end{equation*}%
where $c_{1},c_{2},c_{3}\in \mathbb{R}$, $c_{3}\neq 0.$
\end{theorem}

\begin{theorem}
A $D$-singular minimal translation surface in $\mathbb{R}^{3}$ of type $x=f(y)+g(z)$
with respect to a horizontal vector $\mathbf{u}$ is a generalized cylinder
satisfying $f\left( y\right) =c_{1}$ and 
\begin{equation}
z=\pm \int \left[ c_{2}^{2}\left( c_{1}+g\right) ^{2\alpha }-1\right]
^{-1/2}dg,  \label{4.4}
\end{equation}%
for $c_{1},c_{2}\in \mathbb{R},$ $c_{2}\neq 0.$
\end{theorem}

\begin{proof}
Without of loss of generality we may assume that $\mathbf{u}=\partial _{x}$.
Then Eq. \eqref{2.6} implies 
\begin{equation}
\frac{\left[ 1+\left( g^{\prime }\right) ^{2}\right] f^{\prime \prime }+%
\left[ 1+\left( f^{\prime }\right) ^{2}\right] g^{\prime \prime }}{1+\left(
f^{\prime }\right) ^{2}+\left( g^{\prime }\right) ^{2}}=\frac{\alpha }{f+g},
\label{4.5}
\end{equation}%
in which the roles of $f$ and $g$ are symmetric. Eq. \eqref{4.4} has a
solution provided that $f$ or $g$ is a constant, see \cite[Theorem 4.1]%
{Lopez5}. Thanks to the symmetry, we can assume $f=f_{0}\in \mathbb{R}.$
Thereby, Eq. \eqref{4.5} reduces to 
\begin{equation}
\frac{g^{\prime \prime }}{1+\left( g^{\prime }\right) ^{2}}=\frac{\alpha }{%
f_{0}+g}.  \label{4.6}
\end{equation}%
Put 
\begin{equation}
g^{\prime }=q,\text{ }q^{\prime }=\frac{dq}{dg}\frac{dg}{dz}=\frac{g^{\prime
\prime }}{g^{\prime }},\text{ }q=q\left( g\right) .  \label{4.7}
\end{equation}%
By considering Eq. \eqref{4.7} into Eq. \eqref{4.6} we get%
\begin{equation}
\frac{qq^{\prime }}{1+q^{2}}=\frac{\alpha }{f_{0}+g}.  \label{4.8}
\end{equation}%
A first integration of Eq. \eqref{4.8} with respect to $g$ yields $%
1+q^{2}=c^{2}\left( f_{0}+g\right) ^{2\alpha }$ or%
\begin{equation*}
dz=\pm \frac{dg}{\sqrt{c^{2}\left( f_{0}+g\right) ^{2\alpha }-1}}
\end{equation*}%
and the proof is completed by a first integration.
\end{proof}

\section{Conclusions}

In this study, new perspectives on the singular minimality of immersed surfaces
in $\mathbb{R}^{3}$ were introduced. These, in particular the $\nabla -$%
singular minimality, provided us non-trivial and different results from the
obtained one with respect to the Levi-Civita connection on $\mathbb{R}^{3},$
see \cite{Lopez5,Lopez7}. Our results were found with respect to a
horizontal vector and it will not make  big difference when the vector is assumed to be
vertical. Yet, the problem of finding $\nabla -$ and $D-$%
singular minimal translation surfaces in $\mathbb{R}^{3}$ with respect to an
arbitrary vector is open.

\end{document}